\documentclass[11pt]{amsart}

\usepackage{amsfonts,amssymb,stmaryrd,amscd,amsmath,latexsym,amsbsy}

\newtheorem{theorem}{Theorem}[section]

\newtheorem{proposition}[theorem]{Proposition}
\newtheorem{corollary}[theorem]{Corollary}
\theoremstyle{definition}

\newtheorem{remark}[theorem]{Remark}

\newcommand{\Rep}{\text{Rep}}

\newcommand{\g}{\mathfrak{g}}

\newcommand{\h}{\mathfrak{h}}

\newcommand{\ben}{\begin{enumerate}}
\newcommand{\een}{\end{enumerate}}

\theoremstyle{plain}

\newtheorem*{sol}{Solution}

\theoremstyle{definition}

\theoremstyle{remark}

\newcommand{\solu}[1]{\begin{sol}{\bf (\ref{#1})}}

\def\g{\mathfrak{g}}

\def\h{\mathfrak{h}}

\def\Rep{\mathop{\mathrm{Rep}}\nolimits}

\begin{document}

\title{Mittag-Leffler type sums associated with root systems}

\begin{abstract}This is a largely expository note which applies standard techniques of the theory of Duijstermaat-Heckman measures for compact Lie groups and results of P. Littelmann to prove a generalization of a conjecture of Coquereaux and Zuber.
\end{abstract}

\author{Pavel Etingof}
\address{Department of Mathematics, Massachusetts Institute of Technology,
Cambridge, MA 02139, USA}
\email{etingof@math.mit.edu}

\author{Eric Rains}
\address{Department of Mathematics, California Institute of Technology,
Pasadena, CA 91125, USA}
\email{rains@caltech.edu}

\maketitle 

\section{The main theorem}

Let $G$ be a simply connected simple algebraic group over $\Bbb C$. Let $\g={\rm Lie}G$, $\h\subset \g$ be a Cartan subalgebra, $W$ the Weyl group, $Q,P,Q^\vee,P^\vee$ the root, weight, coroot and coweight lattices, $\rho\in \h^*$ the half-sum of positive roots, $P_+\subset P$ the set of dominant integral weights, $V_\lambda$ the irreducible representation of $G$ with highest weight $\lambda\in P_+$, and $k$ a positive integer. Consider the function on $\h$ given by 
$$
f(x)=\prod_{\alpha>0}\frac{\sin\pi \alpha(x)}{\pi\alpha(x)},
$$
where the product is taken over positive roots of $G$. 

Let $Z\subset G$ be the center and $\xi: Z\to \Bbb C^*$ a character. Since $Z=P^\vee/Q^\vee$, we may view $\xi$ as a character of $P^\vee$ which is trivial on $Q^\vee$. Define the function 
$$
F_{k,\xi}(x):=\sum_{a\in P^\vee}\xi(a)f^k(x+a).
$$
(If $G=SL_2$ and $k=1$ then the sum is not absolutely convergent, and should be understood in the sense of principal value). Thus, the meromorphic function 
$$
M_{k,\xi}(x):={\frac{F_{k,\xi}(x)}{\prod_{\alpha>0}\pi^{-k}\sin^k \pi\alpha(x)}}
$$ 
has a Mittag-Leffler type decomposition 
$$
M_{k,\xi}(x)=\sum_{a\in P^\vee}\frac{(-1)^{2k\rho(a)}\xi(a)}{\prod_{\alpha>0}\alpha^k(x)}.
$$
For example, for $G=SL_2$, we have 
$$
F_{1,1}(x)=1, F_{1,-1}(x)=\cos \pi x, 
$$
which gives the classical Mittag-Leffler decompositions 
$$
\pi\cot \pi x=\sum_{n\in \Bbb Z}\frac{1}{x+n},\ \frac{\pi}{\sin \pi x}=\sum_{n\in \Bbb Z}\frac{(-1)^n}{x+n}. 
$$

The goal of this note is to prove the following theorem. 

\begin{theorem}\label{maint} 
The function $F_{k,\xi}$ is a $W$-invariant trigonometric polynomial on the maximal torus $T=\h/Q^\vee$
of $G$, which is a nonnegative rational linear combination of irreducible characters of $G$. 
\end{theorem} 

For $G=SL_n$ and $\xi$ being a character of order $2$, this theorem was conjectured by R. Coquereaux and J.-B. Zuber (\cite{CZ}, Conjecture 1 in Subsection 2.2). 

Since all the techniques and ideas we use are well known, this note should be viewed as largely expository.  

\section{Proof of the main theorem} 

\subsection {Contraction of representations} 
We start with the following general fact.  

\begin{proposition}\label{p1}(\cite{Li1}, Proposition 3) Let $(V,\rho_V)$ be a rational representation of $G$, and $N$ a positive integer. Let $V_N$ be the direct sum of all the weight subspaces of $V$ of weights divisible by $N$. Then the action of $T$ on $V_N$ given by $t\circ v:=\rho_V(t^{1/N})v$ extends to an action of $G$.\footnote{Note that $\rho_V(t^{1/N})v$ is independent on the choice of the $N$-th root $t^{1/N}$.} In other words, the function 
$$
\chi_{V,N}(x):=\sum_{\lambda\in P}\dim V[N\lambda]e^{2\pi i\lambda(x)}
$$
is a nonnegative linear combination of irreducible characters of $G$. Namely, the multiplicity of 
$\chi_\lambda$ in $\chi_{V,N}(x)$ equals the multiplicity of 
$V_{N\lambda+(N-1)\rho}$ in $V\otimes V_{(N-1)\rho}$. 
\end{proposition} 

\begin{proof} Littelmann proves this proposition via his path model as an illustration of its use, 
but we give a more classical proof using the Weyl character formula. We have to show that the integral
$$
I:=\int_{\h_\Bbb R/Q^\vee}\sum_{\lambda\in P}\dim V[N\lambda]e^{-2\pi i\lambda(x)}\chi_\lambda(x)|\Delta(x)|^2dx
$$
is nonnegative, where $\Delta(x)$ is the Weyl denominator, since the multiplicity in question is $I/|W|$. 

Denoting the character of $V$ by $\chi_V$, we have 
$$
I=\int_{\h_\Bbb R/Q^\vee}\int_{\h_\Bbb R/Q^\vee}\sum_{\lambda\in P}\overline{\chi_V(y)}e^{2\pi iN\lambda(y)}e^{-2\pi i\lambda(x)}\chi_\lambda(x)|\Delta(x)|^2dydx=
$$
$$
\int_{\h_\Bbb R/Q^\vee}\int_{\h_\Bbb R/Q^\vee}\overline{\chi_V(y)}\delta(x-Ny)\chi_\lambda(x)|\Delta(x)|^2dydx=
$$
$$
\int_{\h_\Bbb R/Q^\vee}\overline{\chi_V(y)}\chi_\lambda(Ny)|\Delta(Ny)|^2dy.
$$
Using the Weyl character formula, we then have 
$$
I=\int_{\h_\Bbb R/Q^\vee}\overline{\chi_V(y)}\left(\sum_{w\in W}(-1)^we^{2\pi i(w(\lambda+\rho),Ny)}\right)\overline{\Delta(Ny)}dy=
$$
$$
\int_{\h_\Bbb R/Q^\vee}\overline{\chi_V(y)}\frac{\sum_{w\in W}(-1)^we^{2\pi i(w(\lambda+\rho),Ny)}}{\Delta(y)}\frac{\overline{\Delta(Ny)}}{\overline{\Delta(y)}}|\Delta(y)|^2dy=
$$
$$
\int_{\h_\Bbb R/Q^\vee}\overline{\chi_V(y)}\chi_{N\lambda+(N-1)\rho}(y)\frac{\overline{\Delta(Ny)}}{\overline{\Delta(y)}}|\Delta(y)|^2dy.
$$
Now recall that $\frac{\Delta(Ny)}{\Delta(y)}=\chi_{(N-1)\rho}(y)$. Thus we get 
$$
I=\int_{\h_\Bbb R/Q^\vee}\overline{\chi_V(y)\chi_{(N-1)\rho}(y)}\chi_{N\lambda+(N-1)\rho}(y)|\Delta(y)|^2dy,
$$
i.e., $I/|W|$ is the multiplicity of $V_{N\lambda+(N-1)\rho}$ in $V\otimes V_{(N-1)\rho}$, as desired. 
\end{proof} 

\begin{remark} 1. If $N$ is odd (and coprime to $3$ for $G$ of type G2), then 
Proposition \ref{p1} has a nice representation-theoretic interpretation. Namely, 
if $q$ is a root of unity of order $N$ and $G_q$ the corresponding Lusztig quantum group, 
then there is an exact {\it contraction functor} $F: \Rep G_q\to \Rep G$ which 
at the level of $P$-graded vector spaces transforms $V$ into $V_N$ with weights divided by $N$ 
(see \cite{GK} and references therein). Proposition \ref{p1} is then obtained by applying 
the functor $F$ to a Weyl module. 

2. Suppose that $G$ is not simply laced. Normalize the inner product on $\h^*$ so that long roots 
have squared length $2$. This inner product identifies $\h$ with $\h^*$ so that $\alpha_i^\vee$ map to $2\alpha_i/(\alpha_i,\alpha_i)$. Note that $2/(\alpha_i,\alpha_i)$ is an integer, so under this identification $Q^\vee\subset Q$, hence $P^\vee\subset P$. Let $V_N'\subset V_N$ be the span of the weight subspaces of $V$ of weights belonging to $NP^\vee$ with weights divided by $N$. Then, analogously to Proposition \ref{p1}, $V_N'$ extends to a representation of the Langlands dual Lie algebra $\g^L$, with a similar descrition of multiplicities (\cite{Li1}, Proposition 4). Note that this statement is nontrivial even if $\g^L\cong \g$, since the arrow on the Dynkin diagram is reversed. This also has a representation-theoretic 
interpretation similar to (1), see \cite{Li1}, Section 3, \cite{GK}.

3. As explained in \cite{Li1}, 
Proposition \ref{p1} generalizes to symmetrizable Kac-Moody algebras (both our proof and that of \cite{Li1} can 
be straightforwardly extended to this case). So does the non-simply laced version of Proposition \ref{p1} given in (2) and the above representation-theoretic interpretations, see \cite{Li2}.
\end{remark}

\subsection{Duistermaat-Heckman measures and proof of Theorem \ref{maint}} 
Now recall (\cite{GLS}) that for each dominant $\lambda\in \h_{\Bbb R}^*$ we can define the Duistermaat-Heckman measure $DH_\lambda(\mu)d\mu$ on $\h_{\Bbb R}^*$, which is the direct image of the Liouville measure on the coadjoint orbit of $\lambda$. This measure is supported on the convex hull of the Weyl group orbit $W\lambda$, and its Fourier transform is given by the formula 
\begin{equation}\label{eq1}
{\mathcal F}(DH_\lambda)(x)=\frac{\sum_{w\in W}(-1)^we^{2\pi i(w\lambda,x)}}{\prod_{\alpha>0}2\pi i\alpha(x)}. 
\end{equation} 

For simplicity assume that $\lambda$ is regular and $G\ne SL_2$. Then $DH_\lambda$ is absolutely continuous with respect to the Lebesgue measure (i.e., the density function $DH_\lambda(\mu)$ is continuous). 
Then it is known (\cite{GLS}) that if $\mu_N\in P,\lambda_N\in P_+$ are sequences such that $\mu_N/N\to \mu, \lambda_N/N\to \lambda$ as $N\to \infty$ and $\lambda_N-\mu_N\in Q$ then 
\begin{equation}\label{eq2}
\lim_{N\to \infty}\frac{\dim V_{\lambda_N}[\mu_N]}{N^{|R_+|}}=DH_\lambda(\mu), 
\end{equation}
where $R_+$ is the set of positive roots. 
Note that equation \eqref{eq1} follows immediately from equation \eqref{eq2} and the Weyl character formula. 

\begin{proposition}\label{p2} Let $\lambda_1,...,\lambda_k\in \h_{\Bbb R}^*$ be regular dominant weights. Then the trigonometric polynomial
$$
\sum_{\mu\in P}(DH_{\lambda_1}*...*DH_{\lambda_k})(\mu)e^{2\pi i\mu(x)}
$$
(where $*$ denotes convolution of measures) is a linear combination of irreducible characters of $G$ with nonnegative real coefficients. 
\end{proposition} 

\begin{proof}  First assume that $\lambda_i$ are rational, and let $d$ be their common denominator. 
Then, taking the limit as $N\to \infty$ in Proposition \ref{p1} with 
$V=V_{N\lambda_1}\otimes...\otimes V_{N\lambda_k}$ and $N$ divisible by $d$, we
obtain the desired statement. Now the general case follows from the facts that rational weights are dense 
in $\h_{\Bbb R}$ and $DH_\lambda(\mu)$ is continuous in $\lambda$. 
\end{proof} 

Now Theorem \ref{maint} follows from equation \eqref{eq1} and Proposition \ref{p2} by taking 
$\lambda_1,...,\lambda_k=\rho$ and noting that by the Weyl denominator formula 
\begin{equation}\label{eq3}
\frac{\sum_{w\in W}(-1)^we^{2\pi i(w\rho,x)}}{\prod_{\alpha>0}2\pi i\alpha(x)}=\prod_{\alpha>0}\frac{\sin \pi \alpha(x)}{\pi \alpha(x)}=f(x).
\end{equation} 
The rationality of the coefficients follows from the rationality of the values of the convolution power $(DH_\rho)^{*k}$ at rational points. 

\begin{remark} It follows from \eqref{eq3} that the measure $DH_\rho$ is the convolution of uniform measures on $[-\alpha/2,\alpha/2]$ over all positive roots $\alpha$.  
\end{remark} 

\subsection{The characters occuring in $F_{k,\xi}$.} 

Let us now discuss which irreducible characters can occur in the decomposition of $F_{k,\xi}$. 
Let us view $\xi$ as an element of $P/Q$. 
Clearly, if $\chi_\lambda$ occurs in $F_{k,\xi}$ then the central character of the representation 
$V_\lambda$ must be $\xi=k\rho-\lambda$ mod $Q$. 
If so, then, as shown above, the multiplicity of $\chi_\lambda$ in $F_{k,\xi}$ 
is $(DH_\rho)^{*k}(\lambda)$. Since this density is continuous and supported on the $k$ times dilated convex hull $B$ of the orbit $W\rho$, we see that if $\chi_\lambda$ occurs then $\lambda$ has to be strictly in the interior of $B$. 

Let $m_i(\xi)$ be the smallest strictly positive number such 
that $m_i(\xi)=(\xi,\omega_i^\vee)$ in $\Bbb R/\Bbb Z$, where $\omega_i^\vee$ are 
the fundamental coweights, and let 
$\beta_\xi:=\sum_i m_i(\xi)\alpha_i\in P$. 
Then we get 

\begin{proposition}\label{p3} The character $\chi_\lambda$ occurs in $F_{k,\xi}$  if and only if 
$\xi=k\rho-\lambda$ mod $Q$ and 
$(\lambda,\omega_i^\vee)<k(\rho,\omega_i^\vee)$ for all $i$. Moreover, in  presence of the first condition, 
the second condition is equivalent to the inequality $\lambda\le k\rho-\beta_\xi$. 
\end{proposition} 

We also have 

\begin{proposition} The weight $\rho-\beta_\xi$ (and hence $k\rho-\beta_\xi$ for all $k\ge 1$) is dominant. 
\end{proposition} 

\begin{proof} We need to show that for all $i$ we have $(\rho-\beta_\xi,\alpha_i^\vee)\ge 0$. 
Since $\rho-\beta_\xi$ is integral, it suffices to show that $(\rho-\beta_\xi,\alpha_i^\vee)>-1$, i.e., 
$(\beta_\xi,\alpha_i^\vee)<2$. But $(\beta_\xi,\alpha_i^\vee)=2m_i+\sum_{j\ne i}m_j(\alpha_j,\alpha_i^\vee)<2$, since $0<m_j\le 1$ and $(\alpha_j,\alpha_i^\vee)\le 0$ for all $j\ne i$ and is strictly negative for some $j$. 
\end{proof} 

\begin{corollary} $F_{k,\xi}=\sum_{\mu\le k\rho-\beta_\xi}C_{k,\xi}(\mu)\chi_\mu$,
where $C_{k,\xi}(\mu)\in \Bbb Q_{>0}$. In particular, the leading term is a multiple of $\chi_{k\rho-\beta_\xi}$. 
\end{corollary} 

{\bf Acknowledgements.} P.E. is grateful to J.-B. Zuber for turning his attention to Conjecture 1 of \cite{CZ}.

\end{document}